\documentclass{amsart}

\usepackage{amsmath,amssymb,amsfonts,enumerate,amsthm}

\DeclareMathOperator{\hight}{ht}

\newcommand{\id}{\mbox{id}\,}
\newcommand{\fd}{\mbox{fd}\,}
\newcommand{\rfd}{\mbox{Rfd}\,}

\newcommand{\gfd}{\mbox{Gfd}\,}
\newcommand{\gpd}{\mbox{Gpd}\,}
\newcommand{\gid}{\mbox{Gid}\,}
\newcommand{\CM}{\mbox{CM}\,}
\newcommand{\CMfd}{\mbox{CM$_*$\text{-}fd}\,}

\newcommand{\CMpd}{\mbox{CM$_*$\text{-}pd}\,}
\newcommand{\CMid}{\mbox{CM$_*$\text{-}id}\,}

\newcommand{\pd}{\mbox{pd}\,}

\newcommand{\cmdim}{\mbox{CM-dim}\,}

\newcommand{\gd}{\mbox{G-dim}\,}

\newcommand{\Hom}{\mbox{Hom}\,}

\newcommand{\Ext}{\mbox{Ext}\,}

\newcommand{\Spec}{\mbox{Spec}\,}

\newcommand{\depth}{\mbox{depth}\,}
\newcommand{\width}{\mbox{width}\,}
\newcommand{\grade}{\mbox{grade}\,}
\renewcommand{\dim}{\mbox{dim}\,}

\renewcommand{\H}{\mbox{H}}

\newcommand{\R}{\mathbb{R}}

\newcommand{\fn}{\frak{n}}
\newcommand{\fm}{\frak{m}}
\newcommand{\fp}{\frak{p}}

\newcommand{\fq}{\frak{q}}
\newtheorem{thm}{Theorem}[section]
\newtheorem{cor}[thm]{Corollary}
\newtheorem{lem}[thm]{Lemma}
\newtheorem{prop}[thm]{Proposition}

\newtheorem{defn}[thm]{Definition}

\newtheorem{rem}[thm]{Remark}

\newtheorem{def/not}[thm]{Definition/Notation}

\begin{document}

\bibliographystyle{amsplain}

\date{}

\author{Parviz Sahandi, Tirdad Sharif, and Siamak Yassemi}

\address{(Sahandi) \\Department of Pure Mathematics, Faculty of Mathematical Sciences, University of Tabriz, Tabriz, Iran.}
\email{sahandi@ipm.ir}

\address{(Sharif) \\School of Mathematics, Institute for Research in Fundamental Sciences (IPM), P.O. Box. 19395-5746, Tehran, Iran.}

\email{sharif@ipm.ir}

\address{(Yassemi) \\School of Mathematics, Statistics and Computer Science, University of
Tehran, Tehran, Iran.}

\email{yassemi@ut.ac.ir}



\title{Cohen-Macaulay homological dimensions}

\begin{abstract}

We introduce new homological dimensions, namely the Cohen-Macaulay projective, injective and flat dimensions for homologically bounded complexes. Among other things we show that (a) these
invariants characterize the Cohen-Macaulay property for local rings, (b)
Cohen-Macaulay flat dimension fits between the Gorenstein flat dimension and the large restricted flat dimension, and (c) Cohen-Macaulay injective dimension fits between the Gorenstein injective dimension and the Chouinard invariant.
\end{abstract}

\maketitle

\section{Introduction}

A commutative Noetherian local ring $R$ is
regular if the residue field $k$ has finite projective dimension
and only if all $R$-modules have finite projective dimension \cite{AB1} \cite{S}. This theorem of Auslander, Buchsbaum and Serre is a main motivation of studing homological dimensions. The injective and flat
dimensions have similar behavior.

Auslander and Bridger \cite{AB}, introduced a homological dimension for finitely generated modules
designed to single out modules with properties similar to those of
modules over Gorenstein rings.  They called it $\mbox{G}$-dimension
and it is a refinement  of the projective dimension and showed that
a local Noetherian ring $(R,\fm,k)$ is Gorenstein if the residue
field $k$ has finite $\mbox{G}$-dimension and only if all finitely
generated $R$-modules have finite $\mbox{G}$-dimension.

To extend the $\mbox{G}$-dimension beyond the realm of finitely generated modules
over Noetherian rings, Enochs and Jenda \cite{EJ1} introduced the notion of Gorenstein projective module. Then the notion of Gorenstein projective dimension was studied in \cite{C00}.

The notion of Gorenstein injective module is dual to that of
Gorenstein projective module and were introduced in the same paper
by Enochs and Jenda \cite{EJ1}. Then the notion of Gorenstein injective dimension was studied in \cite{C00}.

Another extension of the $\mbox{G}$-dimension is based on Gorenstein flat modules,
a notion due to Enochs, Jenda, and Torrecillas \cite{EJT}. Then the notion of Gorenstein flat dimension was studied in \cite{C00}.

More recently, the complete intersection dimension has been introduced for finitely generated $R$-modules, using quasi-deformations and projective dimension, to characterize the complete intersection property of local rings \cite{AGP}.  Parallel to Gorenstein projective, injective  and flat dimensions, the complete intersection projective, injective  and flat dimensions have been introduced and studied in \cite{SSY}, \cite{SSY11}, \cite{SW} and \cite{SSY12}.

The \textit{Cohen-Macaulay dimension} of a finitely generated $R$-module $M$, as defined
by Gerko \cite{G} is
\begin{align*}
&\cmdim_R(M):= \\[1ex]
        &\inf\left\{\gd_Q(M\otimes_RR') -\gd_Q(R')\bigg| \begin{array}{l} R\rightarrow R'\leftarrow Q\text{ is a
}\\\CM\text{-quasi-deformation } \end{array} \right\}
\end{align*}
(see Section 2 for the definition of $\CM$-quasi-deformation).

The purpose of this paper is to develop a similar theory of projective, injective  and flat analogue for Cohen-Macaulay case. Thus we introduce Cohen-Macaulay projective dimension $(\CMpd)$, Cohen-Macaulay injective dimension $(\CMid)$ and Cohen-Macaulay flat dimension $(\CMfd)$ for homologically bounded complexes over commutative
Noetherian local rings $(R,\fm,k)$ with identity (see Definition \ref{d}). In particular $\cmdim_R(M)=\CMpd_R(M)=\CMfd_R(M)$, for a finitely generated $R$-module $M$. Among other things, we show that these invariants characterize the Cohen-Macaulay property for local rings. We also show that if $M$ is a homologically bounded $R$-complex, then we have the inequalities $$\rfd_R(M)\leq\CMfd_R(M)\leq\gfd_R(M),$$ with equality to the left of any finite value. In particular if $\gfd_R(M)<\infty$, then $\CMfd_R(M)=\gfd_R(M)$, and if $\CMfd_R(M)<\infty$, then $$\CMfd_R(M)=\sup\{\depth R_{\fp} - \depth_{R_{\fp}}(M_{\fp})\mid \fp\in\Spec(R)\},$$ where $\rfd_R(M)$ is the large restricted flat dimension. Also, we show that there are inequalities $$\sup \{\depth R_\fp - \width _{R_{\fp}} M_\fp | \fp\in \Spec(R)\}\le\CMid_R(M)\le\gid_R(M),$$ such that if $\gid_R(M)<\infty$, then $\CMid_R(M)=\gid_R(M)$, and if $\CMid_R(M)<\infty$ for a homologically finite $R$-complex $M$, then
\begin{align*}
\CMid_R(M)=&\sup\{\depth R_{\fp} - \width_{R_{\fp}}(M_{\fp})\mid \fp\in\Spec(R)\} \\[1ex]
        = & \depth R-\inf(M).
\end{align*}

Finally we compare our Cohen-Macaulay homological dimensions with the homological dimensions of Holm and J{\o}rgension \cite{HJ}.

\section{Definitions and Notations}

Let $(R,\fm,k)$ and $(S,\fn,l)$ be commutative local Noetherian rings.

We work in the derived category $\mathcal{D}(R)$ of complexes of $R$-modules, indexed homologically. A complex $M$ is \emph{homologically bounded} if $\H_i(M) = 0$ for all $|i|\gg0$; and it is \emph{homologically finite} if $\oplus_i\H_i(M)$ is finitely generated. 

Fix $R$-complexes $M$ and $N$. Let $M\otimes_R^{\mathbf{L}}N$ and $\mathbf{R}\Hom_R(M,N)$ denote the left-derived tensor product
and right-derived homomorphism complexes, respectively. Let $\inf(M)$ and $\sup(M)$ denote the  infimum and supremum, respectively, of the set $\{n\in \mathbb{Z} \mid \H_n(M)\neq 0\}$.

\begin{def/not}{\em A homologically finite $R$-complex $M$ is \emph{reflexive} if the
complex $\mathbf{R}\Hom_R(M,R)$ is homologically bounded and the biduality morphism $\delta_M:M\to\mathbf{R}\Hom_R(\mathbf{R}\Hom_R(M,R),R)$ is
an isomorphism in $\mathcal{D}(R)$. Set
$$\gd_R(M):=-\inf(\mathbf{R}\Hom_R(M,R)),$$
if $M$ is reflexive, and $\gd_R(M):=\infty$ otherwise. Set also $\gd_R(0)=-\infty$. This is the G-dimension of Auslander and Bridger \cite{AB} and Yassemi \cite{Y95}.
}
\end{def/not}

\begin{def/not}{\em
An $R$-module $G$ is \emph{$G$-projective} if there exists an exact sequence of $R$-modules
$$X=\cdots\stackrel{\partial^X_2}{\to}P_1\stackrel{\partial^X_1}{\to}P_0\stackrel{\partial^X_0}{\to}
P_{-1}\stackrel{\partial^X_{-1}}{\to}P_{-2}\stackrel{\partial^X_{-2}}{\to}\cdots$$
such that $G\cong \mbox{Coker}(\partial^X_1)$, each $P_i$ is projective, and $\Hom_R(X,Q)$ is exact for each projective $R$-module $Q$.

An $R$-module $G$ is \emph{$G$-flat} if there exists an exact sequence of $R$-modules
$$Y=\cdots\stackrel{\partial^Y_2}{\to}F_1\stackrel{\partial^Y_1}{\to}F_0\stackrel{\partial^Y_0}{\to}
F_{-1}\stackrel{\partial^Y_{-1}}{\to}F_{-2}\stackrel{\partial^Y_{-2}}{\to}\cdots$$
such that $G\cong \mbox{Coker}(\partial^Y_1)$, each $F_i$ is flat, and $I\otimes_RY$ is exact for each injective $R$-module $I$.

An $R$-module $G$ is \emph{$G$-injective} if there exists an exact sequence of $R$-modules
$$Z=\cdots\stackrel{\partial^Z_2}{\to}I_1\stackrel{\partial^Z_1}{\to}I_0\stackrel{\partial^Z_0}{\to}I_{-1}
\stackrel{\partial^Z_{-1}}{\to}I_{-2}\stackrel{\partial^Z_{-2}}{\to}\cdots$$
such that $G\cong \mbox{Coker}(\partial^Z_1)$, each $I_i$ is injective, and $\Hom_R(I,Z)$ is exact for each injective $R$-module $I$.

Let $M$ be a homologically bounded $R$-complex. A $G$-projective resolution of $M$ is an isomorphism $H\simeq M$ in $\mathcal{D}(R)$ where
$H$ is a complex of $G$-projective $R$-modules such that $H_i = 0$ for all $i \ll 0$. The $G$-projective dimension of $M$ is
$$\gpd_R(M):=\inf\{\sup\{n\mid H_n\neq0\}\mid H\simeq M\text{ is a }G\text{-projective resolution}\}.$$

The $G$-flat dimension of $M$ is defined similarly and denoted $\gfd_R(M)$, while the $G$-injective dimension $\gid_R(M)$ is dual \cite{C00}. These are the $G$-projective, $G$-flat, and $G$-injective dimensions of
Enochs, Jenda and Torrecillas (which they consider only in the case of modules) \cite{EJ1} and \cite{EJT}.
}
\end{def/not}

\begin{rem}\label{2.3}{\em (1) It is known that, for a homologically bounded $R$-complex $M$, $\gpd_R(M)$ and $\gfd_R(M)$ are simultaneously finite \cite[Proposition 4.3]{SW}.

(2) Let $R\to S$ be a flat local homomorphism and $M$ a finitely generated $R$-module. Then it is well-known that, $\gd_R(M)=\gd_S(M\otimes_RS)$ and  $\gd_R(M)=\gfd_R(M)=\gpd_R(M)$ \cite{CFH}.

(3) The finiteness of $G$-projective, $G$-flat, and $G$-injective dimensions characterize the Gorenstein property of local rings \cite{C00}.  }
\end{rem}

\begin{def/not}{\em A finitely generated $R$-module $M$ is called \textit{G-perfect} if $\gd_RM=\grade_R M:=\inf\{i\mid\Ext^i_R(M,R)\neq0\}$. Let $Q$ be a local ring and $J$ an ideal of $Q$.
By abuse of language we say that $J$ is \textit{G-perfect} if the $Q$-module $Q/J$ has the corresponding
property.

A \emph{$\CM$-deformation} of $R$ is a surjective local homomorphism $Q\to R$ such that $J=\ker(Q\to R)$ is a G-perfect ideal  in $Q$. A
\emph{$\CM$-quasi-deformation} of $R$ is
a diagram of local homomorphisms $R\rightarrow R'\leftarrow Q$,
with $R\rightarrow R'$ a flat extension and $R'\leftarrow Q$ a
$\CM$-deformation.

The \textit{Cohen-Macaulay dimension} of a nonzero finitely generated $R$-module $M$, as defined
by Gerko \cite{G} is
\begin{align*}
&\cmdim_R(M):= \\[1ex]
        &\inf\left\{\gd_Q(M\otimes_RR') -\gd_Q(R')\bigg| \begin{array}{l} R\rightarrow R'\leftarrow Q\text{ is a
}\\\CM\text{-quasi-deformation } \end{array} \right\},
\end{align*}
and set $\cmdim_R(0)=-\infty$.
}
\end{def/not}
\begin{rem}\label{2.5}{\em By \cite[Theorems 3.8 and 3.9, and Proposition 3.10]{G} we have

(1) $R$ is Cohen-Macaulay if and only if $\cmdim_R(k)<\infty$.

(2) If $M$ is a finitely generated $R$-module such that $\cmdim_R(M)<\infty$, then $$\cmdim_R(M)=\depth R-\depth_R(M).$$

(3) For each prime ideal $\fp$ of $R$, $\cmdim_{R_{\fp}}(M_{\fp})\le\cmdim_R(M)$.
}
\end{rem}

\begin{def/not}{\em A finitely generated $R$-module $C$ is \emph{semidualizing} if the homothety morphism $\chi_C^R:R\to \mathbf{R}\Hom_R(C,C)$ is an isomorphism in $\mathcal{D}(R)$. A finitely generated $R$-module $D$ is \emph{canonical} if it is semidualizing and $\id_R(D)$ is finite.}
\end{def/not}

Let $\varphi:R\to S$ be a local ring homomorphism. We denote $\widehat{R}$ the completion of $R$ at its maximal
ideal and let $\varepsilon_R:R\to \widehat{R}$ denote the natural map. The \emph{completion of $\varphi$} is the unique local ring homomorphism $\widehat{\varphi}:\widehat{R}\to \widehat{S}$ such that $\widehat{\varphi}\circ\varepsilon_R=\varepsilon_S\circ\varphi$. The \emph{semi-completion of $\varphi$} is the composition $\varepsilon_S\circ\varphi:R\to \widehat{S}$.

\section{Cohen-Macaulay projective, flat and injective dimensions}

In this section we introduce a Cohen-Macaulay projective
dimension, Cohen-Macaulay flat dimension, and
Cohen-Macaulay injective dimension for homologically bounded $R$-complexes and derive their basic properties. When $M$ is a module, Definition \ref{d} is from \cite{SSY}, which is in turn modeled on \cite{AGP} and \cite{G}.

\begin{defn}\label{d}{\em Let $(R,\fm)$ be a local ring. For each homologically bounded $R$-complex $M$, define the Cohen-Macaulay
projective dimension, Cohen-Macaulay flat dimension and Cohen-Macaulay injective dimension of $M$ as,
$$\CMpd_R(M):=\inf\left\{\gpd_Q(M\otimes_RR') -\gfd_Q(R')\bigg| \begin{array}{l} R\rightarrow R'\leftarrow Q\text{ is a
}\\\CM\text{-quasi-deformation} \end{array} \right\}$$
$$\CMfd_R(M):=\inf\left\{\gfd_Q(M\otimes_RR') -\gfd_Q(R')\bigg| \begin{array}{l} R\rightarrow R'\leftarrow Q\text{ is a
}\\\CM\text{-quasi-deformation} \end{array} \right\}$$
$$\CMid_R(M):=\inf\left\{\gid_Q(M\otimes_RR') -\gfd_Q(R')\bigg| \begin{array}{l} R\rightarrow R'\leftarrow Q\text{ is a
}\\\CM\text{-quasi-deformation} \end{array} \right\}$$
respectively.
}
\end{defn}

\begin{rem}\label{3.2}{\em (1) It is known that $\gpd_R(M)$ and $\gfd_R(M)$ are simultaneously finite by Remark \ref{2.3}(1). Hence $\CMpd_R(M)$ and $\CMfd_R(M)$ are simultaneously finite.

(2) By taking the trivial $\CM$-quasi-deformation $R\rightarrow
R\leftarrow R$, one has $$\CMpd_R(M)\le\gpd_R(M),$$ $$\CMfd_R(M)\le\gfd_R(M),$$ $$\CMid_R(M)\le\gid_R(M).$$

(3) By Remark \ref{2.3}(2) it can be seen that if $M$ is a finitely generated $R$-module then, $\CMpd_R(M)=\CMfd_R(M)=\mbox{CM}$-$\dim_R(M)$.
}
\end{rem}

The following two theorems show that the finiteness of these dimensions characterize the Cohen-Macaulay rings.

\begin{thm} \label{CHAR} The following conditions are equivalent:
\begin{itemize}
\item[(1)] The ring $R$ is Cohen-Macaulay.
\item[(2)] $\CMpd_R(M)<\infty$ for every homologically bounded $R$-complex $M$.
\item[(3)] $\CMpd_R(k)<\infty$.
\item[(4)] $\CMfd_R(M)<\infty$ for every homologically bounded $R$-complex $M$.
\item[(5)] $\CMfd_R k<\infty$.
\end{itemize}
\end{thm}
\begin{proof} (1)$\Rightarrow$(2) Let $\widehat{R}$ be
the $\fm$-adic completion of $R$. Since $R$ is Cohen-Macaulay, so
is $\widehat{R}$. Therefore by Cohen's structure theorem,
$\widehat{R}$ is isomorphic to $Q/J$, where $Q$ is a regular local
ring. By Cohen-Macaulay-ness of $\widehat{R}$ and regularity of
$Q$, the ideal $J$ is G-perfect.  Thus $R\rightarrow
\widehat{R}\leftarrow Q$ is a $\CM$-quasi-deformation. Since $Q$
is regular $\gpd_Q(M\otimes_R\widehat{R})<\infty$ for every homologically bounded $R$-complex $M$. Thus
$\CMpd_R(M)$ is finite.

(2)$\Rightarrow$(3) and (4)$\Rightarrow$(5) are trivial.

(2)$\Rightarrow$(4) and (3)$\Rightarrow$(5) are trivial since $\CMfd_R(M)\le\CMpd_R(M)$.

(5)$\Rightarrow$(1) It follows from Remark \ref{3.2}(3) that $\cmdim_R(k)=\CMfd_R(k)<\infty$. Now Remark \ref{2.5}(1), completes the proof.
\end{proof}

\begin{thm}\label{OH} The following conditions are equivalent.
\begin{itemize}
\item[(1)] The ring $R$ is Cohen-Macaulay.
\item[(2)] $\CMid_R(M)<\infty$ for every homologically bounded $R$-complex $M$.
\item[(3)] $\CMid_R(k)<\infty$.
\end{itemize}
\end{thm}

\begin{proof} (1)$\Rightarrow$(2) is the same as proof of part (1)$\Rightarrow$(2) of Theorem \ref{CHAR}.

(2)$\Rightarrow$(3) is trivial.

(3)$\Rightarrow$(1) Suppose $\CMid_R(k)<\infty$. So that there
 is a $\CM$-quasi-deformation $R\rightarrow R'\leftarrow Q$,
 such that $\gid_Q(k\otimes_RR')$ is finite. It is clear that $k\otimes_RR'$
 is a cyclic $Q$-module. Consequently $Q$ is a
Gorenstein ring by \cite[Theorem 4.5]{FF}. We plan to show that $R'$ is a
 Cohen-Macaulay ring. Let $I=\ker(Q\rightarrow R')$ which is G-perfect by definition.
 We have
\begin{align*}
\hight I= &\grade(I,Q) \\[1ex]
        = & \gd_QR' \\[1ex]
        = & \depth Q-\depth_QR' \\[1ex]
        = & \depth Q-\depth R' \\[1ex]
        = & \dim Q-\depth R' \\[1ex]
        = & \hight I+\dim R'-\depth R',
\end{align*}
in which the equalities follow from Cohen-Macaulay-ness of $Q$;
G-perfectness of $I$; Auslander-Buchsbaum formula;
\cite[Exercise 1.2.26]{BH}; Cohen-Macaulay-ness of $Q$; and \cite[Corollary 2.1.4]{BH} respectively. Therefore we obtain that $\dim R'-\depth
R'=0$, that is $R'$ is Cohen-Macaulay. Now \cite[Theorem
2.1.7]{BH} gives us the desired result.
\end{proof}

The proof of the above theorem says some thing more, viz., a local
ring $R$ is Cohen-Macaulay if and only if there exists a cyclic
$R$-module of finite Cohen-Macaulay injective dimension.

\begin{cor} Assume that $C\neq0$ is a cyclic $R$-module. Then $R$ is a Cohen-Macaulay ring if and only if
$\CMid_RC<\infty$.
\end{cor}

\begin{rem}\label{gid}{\em Let $M$ be a homologically finite $R$-complex such that $\gid_R(M)<\infty$. Then by \cite[Theorem 3.6]{FF}, we obtain that $\gid_{\widehat{R}}(M\otimes_R\widehat{R})<\infty$. Hence using \cite[Corollary 2.3]{CS10}, we have
$$\gid_{\widehat{R}}(M\otimes_R\widehat{R})=\depth \widehat{R}-\inf(M\otimes_R\widehat{R})=\depth R-\inf(M)=\gid_R(M).$$
 }
\end{rem}

\begin{prop} \label{ii} Let $M$ be a homologically finite $R$-complex. Then
$$\CMid_R(M)=\inf\left\{\gid_Q(M\otimes_RR') -\gfd_Q(R')\bigg| \begin{array}{l} R\rightarrow R'\leftarrow Q\text{ is a
}\\\CM\text{-quasi-deformation}\\\text{such that }Q\text{ is complete} \end{array} \right\}.$$
\end{prop}

\begin{proof} It is clear that the left hand side is less than or equal to the right hand side. Now let $R\rightarrow R'\leftarrow Q$ be a $\CM$-quasi-deformation. Then note that $R\rightarrow \widehat{R'}\leftarrow \widehat{Q}$ is also a $\CM$-quasi-deformation such that
$$\gid_Q(M\otimes_RR')=\gid_{\widehat{Q}}(M\otimes_RR'\otimes_Q\widehat{Q})=\gid_{\widehat{Q}}(M\otimes_R\widehat{R'}),$$
and $\gfd_Q(R')=\gfd_{\widehat{Q}}(\widehat{R'})$, where the first equality holds by Remark \ref{gid}. So we can assume in the $\CM$-quasi-deformation $R\rightarrow R'\leftarrow Q$ that, $Q$ is a complete local ring. This shows the equality.
\end{proof}

\begin{prop} \label{f} Let $M$ be a homologically bounded $R$-complex. Then
$$\CMfd_R(M)=\inf\left\{\gfd_Q(M\otimes_RR') -\gfd_Q(R')\bigg| \begin{array}{l} R\rightarrow R'\leftarrow Q\text{ is a
}\\\CM\text{-quasi-deformation}\\\text{such that }Q\text{ is complete} \end{array} \right\}.$$
\end{prop}

\begin{proof} The proof is the same as proof of Proposition \ref{ii}, but here use \cite[Corollary 8.9]{IS} instead of Remark \ref{gid}.
\end{proof}

Let $M$ be homologically bounded $R$-complex. Then Foxby showed that $$\gpd_{\widehat{R}}(M\otimes_R\widehat{R})\leq\gpd_R(M)$$ (see \cite[Ascent table II(b)]{ch09}).

\begin{prop} \label{p} Let $M$ be a homologically bounded $R$-complex. Then
$$\CMpd_R(M)=\inf\left\{\gpd_Q(M\otimes_RR') -\gfd_Q(R')\bigg| \begin{array}{l} R\rightarrow R'\leftarrow Q\text{ is a
}\\\CM\text{-quasi-deformation}\\\text{such that }Q\text{ is complete} \end{array} \right\}.$$
\end{prop}

\begin{proof} The proof is the same as proof of Proposition \ref{ii}, but here use the comment just before the proposition instead of Remark \ref{gid}.
\end{proof}

A homological dimension should not grow under localization. Let $\fp$ be a prime ideal of $R$ and $M$ a homologically bounded $R$-complex. It is well known that $$\gfd_{R_{\fp }}(M_{\fp})\leq\gfd_R(M),$$ and Foxby showed that (when $R$ has finite Krull dimension) $$\gpd_{R_{\fp }}(M_{\fp})\leq\gpd_R(M)$$ (see \cite[Page 262]{CFH}). On the other hand if $R$ has a dualizing complex then, $$\gid_{R_{\fp }}(M_{\fp})\leq\gid_R(M)$$ by \cite[Proposition 5.5]{CFH}.

\begin{thm}\label{zzz} Let $M$ be a homologically finite $R$-complex. For each prime ideal $\fp\in\Spec(R)$ there is an inequality
$$\CMid_{R_{\fp }}{(M_{\fp})}\le\CMid_R(M).$$
\end{thm}

\begin{proof} Assume that $\CMid_R(M)<\infty$. Let $R\rightarrow R'\leftarrow Q$ be a $\CM$-quasi-deformation with $Q$ a complete local ring, such that $\gid_Q(M\otimes R')<\infty$ and $\CMid_R(M)=\gid_Q(M\otimes R')-\gfd_Q(R')$ by Proposition \ref{ii}. Hence $Q$ admits a dualizing complex.

Let $\fp$ be a prime ideal of $R$. Since $R\rightarrow
R'$ is a faithfully flat extension of rings, there is a prime ideal
$\fp '$ in $R'$ lying over $\fp$. Let $\fq$ be the inverse image of
$\fp '$ in $Q$. The map $R_{\fp}\rightarrow R'_{\fp '}$ is flat, and
$R'_{\fp'}\leftarrow Q_{\fq}$ is a $\CM$-deformation and note that $\gfd_{Q_{\fq}}(R'_{\fp'})=\gfd_QR'$. Therefore the
diagram $R_{\fp}\rightarrow R'_{\fp'}\leftarrow Q_{\fq}$ is a
$\CM$-quasi-deformation with $$\gid_{Q_{\fq}}(M_{\fp}\otimes_{R_{\fp}} R'_{\fp '})=\gid_{Q_{\fq}}((M\otimes_RR')\otimes_QQ_{\fq})\le\gid_Q(M\otimes_RR')<\infty,$$ where the inequality holds by \cite[Proposition 5.5]{CFH}. Hence $\CMid_{R_{\fp}}(M_{\fp})<\infty$. So
we obtain
\begin{align*}
\CMid_{R_{\fp}}(M_{\fp})\leq & \gid_{Q_{\fq}}(M_{\fp}\otimes_{R_{\fp}} R'_{\fp '})-\gfd_{Q_{\fq}}(R'_{\fp'}) \\[1ex]
                       \le & \gid_Q(M\otimes_RR')-\gfd_Q(R') \\[1ex]
                       = & \CMid_R(M).
\end{align*}
Thus the desired inequality follows.
\end{proof}

We do not know when the inequality $\CMid_{R_{\fp }}{(M_{\fp})}\le\CMid_R(M)$ holds in general. However for $\CMpd_R(M)$ and $\CMfd_R(M)$ we have

\begin{thm}\label{P3} Let $M$ be a homologically bounded $R$-complex. For each prime ideal $\fp\in\Spec(R)$ there is an inequality
\begin{itemize}
  \item[(1)] $\CMpd_{R_{\fp }}{(M_{\fp})}\le\CMpd_R(M)$.
  \item[(2)] $\CMfd_{R_{\fp }}{(M_{\fp})}\le\CMfd_R(M)$.
\end{itemize}
\end{thm}

\begin{proof} The proof is the same as proof of Theorem \ref{zzz}, but here we do not need $Q$ is a complete local ring.
\end{proof}

\begin{prop} \label{iart} Let $M$ be a homologically finite $R$-complex. Then there is an equality
$$\CMid_R(M)=\inf\left\{\gid_Q(M\otimes_RR') -\gfd_Q(R')\bigg| \begin{array}{l} R\rightarrow R'\leftarrow Q\text{ is a
}\\\CM\text{-quasi-deformation}\\\text{such that the closed fibre}\\\text{of }R\to R'\text{ is Artinian} \end{array} \right\}.$$
\end{prop}

\begin{proof}
It is clear that the left hand side is less than or equal to the right hand side. Let $R\rightarrow R'\leftarrow Q$ be a $\CM$-quasi-deformation with $Q$ a complete local ring, such that $\CMid_R(M)=\gid_Q(M\otimes R')-\gfd_Q(R')$ by Proposition \ref{ii}. Hence $Q$ admits a dualizing complex. Now
choose $\fp'\in\Spec(R')$ such that it is a minimal prime ideal
containing $\fm R'$; thus $\fm=\fp'\cap R$ and $\fp'=\fq/J$ for some
$\fq\in\Spec(Q)$, where $J=\ker(Q\rightarrow R')$. Now the diagram
$R\rightarrow R'_{\fp'}\leftarrow Q_{\fq}$ is a
$\CM$-quasi-deformation such that the closed fiber of $R\rightarrow R'_{\fp'}$ is Artinian.  It is clear that $\gfd_Q
R'=\gfd_{Q_{\fq}} R'_{\fp'}$.  Also we have
\begin{align*}
\gid_{Q_{\fq}} (M\otimes_R R'_{\fp'})=&\gid_{Q_{\fq}} (M\otimes_R (R'\otimes_Q Q_{\fq})) \\[1ex]
                       =&\gid_{Q_{\fq}} ((M\otimes_RR')\otimes_Q Q_{\fq}) \\[1ex]
                       \le &\gid_Q(M\otimes_RR'),
\end{align*}
where the inequality holds by \cite[Proposition 5.5]{CFH}.
Hence
$\gid_{Q_{\fq}} (M\otimes_R R'_{\fp})-\gfd_{Q_{\fq}}(R'_{\fp})\le\CMid_R(M)$. So the proof is complete.
\end{proof}

\begin{prop} \label{art} Let $M$ be a homologically bounded $R$-complex. Then there are equalities
$$\CMpd_R(M)=\inf\left\{\gpd_Q(M\otimes_RR') -\gfd_Q(R')\bigg| \begin{array}{l} R\rightarrow R'\leftarrow Q\text{ is a
}\\\CM\text{-quasi-deformation}\\\text{such that the closed fibre}\\\text{of }R\to R'\text{ is Artinian} \end{array} \right\},$$
$$\CMfd_R(M)=\inf\left\{\gfd_Q(M\otimes_RR') -\gfd_Q(R')\bigg| \begin{array}{l} R\rightarrow R'\leftarrow Q\text{ is a
}\\\CM\text{-quasi-deformation}\\\text{such that the closed fibre}\\\text{of }R\to R'\text{ is Artinian} \end{array} \right\}.$$
\end{prop}

\begin{proof} The proof is the same as proof of Proposition \ref{iart}, but here we do not need $Q$ is a complete local ring.
\end{proof}

\begin{rem} \label{iar}{\em (1) Let $M$ be a homologically finite $R$-complex. Then, one can combine the proofs of Propositions \ref{ii} and \ref{iart}, to obtain an equality
$$\CMid_R(M)=\inf\left\{\gid_Q(M\otimes_RR') -\gfd_Q(R')\bigg| \begin{array}{l} R\rightarrow R'\leftarrow Q\text{ is a
}\\\CM\text{-quasi-deformation}\\\text{such that }Q\text{ is complete}\\\text{and the closed fibre of}\\R\to R'\text{ is Artinian} \end{array} \right\}.$$
(2) Likewise for a homologically bounded $R$-complex $M$, one can combine the proofs of Propositions \ref{f}, \ref{p} and \ref{iart}, to obtain the equalities
$$\CMfd_R(M)=\inf\left\{\gfd_Q(M\otimes_RR') -\gfd_Q(R')\bigg| \begin{array}{l} R\rightarrow R'\leftarrow Q\text{ is a
}\\\CM\text{-quasi-deformation}\\\text{such that }Q\text{ is complete}\\\text{and the closed fibre of}\\R\to R'\text{ is Artinian} \end{array} \right\},$$
$$\CMpd_R(M)=\inf\left\{\gpd_Q(M\otimes_RR') -\gfd_Q(R')\bigg| \begin{array}{l} R\rightarrow R'\leftarrow Q\text{ is a
}\\\CM\text{-quasi-deformation}\\\text{such that }Q\text{ is complete}\\\text{and the closed fibre of}\\R\to R'\text{ is Artinian} \end{array} \right\}.$$
}
\end{rem}

\section{Large restricted flat dimension and Chouinard's invariant}

Recall from \cite{CFF}, that the \emph{large restricted flat dimension} is defined by
$$\rfd_R(M):=\sup\{\sup(F\otimes_R^{\mathbf{L}} M)\mid F\text{ an }R\text{-module with }\fd_R(F)<\infty \}.$$ This number is finite, as long as $\H(M)$ is nonzero and the Krull dimension of $R$ is finite; see \cite[Proposition 2.2]{CFF}. It is useful to keep in mind an alternative formula
\cite[Theorem 2.4]{CFF} for computing this invariant:
$$\rfd_R(M)=\sup\{\depth R_{\fp} - \depth_{R_{\fp}}(M_{\fp})\mid \fp\in\Spec(R)\}.$$ Recall here that the \emph{depth} of a homologically bounded $R$-complex $M$ is defined by $$\depth_R(M)=-\sup(\mathbf{R}\Hom_R(k,M)),$$ and it is shown that $\depth_R(M)\ge-\sup(M)$.

It is proved in \cite[Theorem 8.8]{IS} that for an $R$-complex $M$,
$\rfd_R(M)$ is a refinement of $\gfd_R(M)$, that is
$$\rfd_R(M)\le\gfd_R(M),$$ with equality if $\gfd_R(M)$ is finite.

First, we plan to show that, when the Cohen-Macaulay flat dimension of a homologically bounded $R$-complex $M$ is finite, then it is equal to the large restricted flat dimension of $M$. The following proposition is the main tool.

\begin{prop} \label{R1} Let $R\rightarrow S\leftarrow Q$ be a $\CM$-quasi-deformation,
and let $M$ be a homologically bounded $R$-complex. Then
$$\rfd_R(M)=\rfd_Q (M\otimes_R S)-\rfd_Q(S).$$
\end{prop}
\begin{proof} First we prove the equality $$\rfd_S(Y)=
\rfd_Q(Y)-\gd_Q(S),$$ for a homologically bounded $S$-complex $Y$. To this end, choose by \cite[Theorem 2.4(b)]{CFF} a prime ideal $\fp$ of $S$ such that the first
equality below holds. Let $\fq$ be the inverse image of $\fp$ in
$Q$. Therefore there is an isomorphism $Y_{\fp}\cong Y_{\fq}$ of
$Q_{\fq}$-modules and a $\CM$-deformation $Q_{\fq}\rightarrow
S_{\fp}$.  Hence
\begin{align*}
\rfd_S(Y)= & \depth S_{\fp}-\depth_{S_{\fp}}Y_{\fp} \\[1ex]
        = & \depth_{Q_{\fq}} S_{\fp}-\depth_{Q_{\fq}}Y_{\fp} \\[1ex]
        = & \depth Q_{\fq}-\gd_{Q_{\fq}} S_{\fp}-\depth_{Q_{\fq}}Y_{\fp} \\[1ex]
        \le & \rfd_Q(Y)-\gd_{Q_{\fq}}(S_{\fp}) \\[1ex]
        = & \rfd_Q(Y)-\gd_Q(S).
\end{align*}
The second equality holds since $Q_{\fq}\rightarrow S_{\fp}$ is
surjective and \cite[Proposition 5.2(1)]{i99}; the third equality holds by Auslander-Bridger formula
\cite{AB}; the fourth equality is due to the
$\mbox{G}$-perfectness assumption of $S$ over $Q$; while the
inequality follows from \cite[Theorem 2.4(b)]{CFF}. Now by
\cite[Proposition 3.5]{SY2} we have
$$\rfd_Q(Y)\le\rfd_S(Y)+\rfd_Q(S)\le\rfd_Q(Y)-\gd_Q(S)+\rfd_Q(S)=\rfd_Q(Y),$$
which is the desired equality.

Now we have
\begin{align*}
\rfd_Q (M\otimes_R S)\le & \rfd_S (M\otimes_R S)+\rfd_Q(S) \\[1ex]
                     = & \rfd_S (M\otimes_R S)+\gd_Q(S) \\[1ex]
                     = & \rfd_Q (M\otimes_R S),
\end{align*}
where the inequality is in \cite[Proposition 3.5]{SY2}, the first
equality follows from the hypotheses, and the second equality follows from the
above observation.  Hence $$\rfd_Q (M\otimes_R S)-\rfd_Q(S)=\rfd_S(M\otimes_R S)=\rfd_R(M)$$ where the second equality holds by
\cite[Lemma 8.5(1)]{IS}.
\end{proof}

\begin{cor}\label{rfdcmfd} Let $M$ be a homologically bounded $R$-complex. Then we have the inequalities $$\rfd_R(M)\leq\CMfd_R(M)\leq\gfd_R(M),$$ with equality to the left of any finite value. In particular if $\CMfd_R(M)<\infty$, then
\begin{align*}
\CMfd_R(M)=&\sup\{\depth R_{\fp} - \depth_{R_{\fp}}(M_{\fp})\mid \fp\in\Spec(R)\} \\[1ex]
          \le & \dim R+\sup(M).
\end{align*}
\end{cor}

Now using Corollary \ref{rfdcmfd}, we investigate the effect of change of ring
on Cohen-Macaulay flat dimension.

\begin{prop}\label{P1} Let $M$ be a homologically bounded $R$-complex. Let $R\rightarrow R'$ be a local flat extension, and $M'=M\otimes_R
R'$. Then $$\CMfd_R(M)\le\CMfd_{R'}(M')$$ with equality when $\CMfd_{R'}(M')$ is finite.
\end{prop}

\begin{proof} Suppose that $\CMfd_{R'}(M')<\infty$, and let $R'\rightarrow
R''\leftarrow Q$ be a $\CM$-quasi-deformation with $\gfd_Q(M'\otimes_{R'}R'')<\infty$. Since $R\rightarrow R'$
and $R'\rightarrow R''$ are flat extensions, the local homomorphism
$R\rightarrow R''$ is also flat. Hence $R\rightarrow R''\leftarrow
Q$ is a $\CM$-quasi-deformation with $\gfd_Q (M\otimes_R
R'')<\infty$. It follows that $\CMfd_R(M)$ is finite. Now by Corollary \ref{rfdcmfd} and \cite[Lemma 8.5(1)]{IS}, we have
$$\CMfd_R(M)=\rfd_R(M)=\rfd_{R'}(M')=\CMfd_{R'}(M'),$$
to complete the proof.
\end{proof}

\begin{prop}\label{P2}
For every homologically bounded $R$-complex $M$
$$\CMfd_R(M)=\CMfd_{\widehat{R}}(M\otimes_R \widehat{R}).$$
\end{prop}

\begin{proof} If $\CMfd_R(M)=\infty$, then we obtain that
$\CMfd_{\widehat{R}}(M\otimes_R \widehat{R})=\infty$ by Proposition \ref{P1}.
Now assume that $\CMfd_R(M)<\infty$. Using Proposition \ref{P1}, it is sufficient to prove
that $\CMfd_{\widehat{R}}(M\otimes_R \widehat{R})$ is finite. To this end, choose a $\CM$-quasi-deformation $R\rightarrow R'\leftarrow Q$ of
$R$ such that $\gfd_Q(M\otimes_RR')<\infty$. So we have $\widehat{R}\rightarrow \widehat{R'}\leftarrow
\widehat{Q}$ is a $\CM$-quasi-deformation of $\widehat{R}$ with
respect to their maximal ideal-adic completions. Now using \cite[Corollary 8.9]{IS} we obtain
$$
\gfd_{\widehat{Q}} ((M\otimes_R
\widehat{R})\otimes_{\widehat{R}} \widehat{R'})=\gfd_Q (M\otimes_R R')<\infty.
$$
Hence $\gfd_{\widehat{Q}} ((M\otimes_R
\widehat{R})\otimes_{\widehat{R}} \widehat{R'})$ is finite which in turn
implies that $\CMfd_{\widehat{R}}(M\otimes_R \widehat{R})$ is
finite.
\end{proof}

Next, recall that the \emph{width} of an $R$-complex $M$ is defined by $$\width_R(M)=\inf(M\otimes_R^{\mathbf{L}}k),$$ and that $\width_R(M)\ge\inf(M)$. Also, if $M$ is homologically finite, then $$\width_R(M)=\inf(M).$$ It is the
dual notion for $\depth_R(M)$. In particular by \cite[Proposition 4.8]{CFF}, we
have $$\width_R(M)= \depth_R(\mathbf{R}\Hom_R(M,\mbox{E}_R(k))),$$ where $\mbox{E}_R(k)$ denotes
the injective envelope of $k$ over $R$.

The \emph{Chouinard invariant} \cite[Corollary 3.1]{C} is denoted by
$\mbox{Ch}_R(M)$ and
$$\mbox{Ch}_R(M):=\sup \{\depth R_\fp - \width _{R_{\fp}}(M_\fp) | \fp\in \Spec(R)\}.$$

It is proved in \cite[Theorem 2.2]{CS10} that for an $R$-complex $M$,
$\mbox{Ch}_R(M)$ is a refinement of $\gid_R(M)$, that is
$$\mbox{Ch}_R(M)\le\gid_R(M),$$ with equality if $\gid_R(M)$ is finite. Now we want to show that the Cohen-Macaulay injective dimension is bounded below by the Chouinard's invariant.

\begin{lem} \label{CH2} Suppose that $Q\rightarrow S$ is a surjective local homomorphism and $Y$ is an
$S$-complex. Then we have
$$\width_S(Y)=\width_Q(Y).$$
\end{lem}
\begin{proof} We have the following equalities:
\begin{align*}
\width_S(Y)= & \depth_S\mathbf{R}\Hom_S(Y,\mbox{E}_S(k)) \\[1ex]
         = & \depth_S\mathbf{R}\Hom_S(Y,\Hom_Q(S,\mbox{E}_Q(k))) \\[1ex]
         = & \depth_S\mathbf{R}\Hom_Q(Y,\mbox{E}_Q(k)) \\[1ex]
         = & \depth_Q\mathbf{R}\Hom_Q(Y,\mbox{E}_Q(k)) \\[1ex]
         = & \width_Q(Y),
\end{align*}
where the first one is by \cite[Proposition 4.8]{CFF}; the second one is by
\cite[Lemma 10.1.15]{BS}; the third one is by adjointness of $\Hom$
and tensor; the fourth one is true since $Q\rightarrow S$ is
surjective and \cite[Proposition 5.2(1)]{i99}; while the last one is again by \cite[Proposition 4.8]{CFF}. Here
we used $k$ for the residue fields of $Q$ and $S$, and $\mbox{E}_Q(k)$
and $\mbox{E}_S(k)$ for the injective envelopes of $k$ over respectively
$Q$ and $S$.
\end{proof}

\begin{lem} \label{CH41} Suppose that $R\rightarrow S$ is a flat local ring homomorphism, and $M$ is a homologically bounded
$R$-complex. Then we have
$$\width_S(M\otimes_RS)=\width_R(M).$$
\end{lem}
\begin{proof} A standard application of the K\"{u}nneth formula yields the equality.
\end{proof}

\begin{prop} \label{CH4} Let $R\rightarrow S$ be a flat local homomorphism and let $M$ be a homologically bounded
$R$-complex. Then
$$\mbox{Ch}_R(M)\le \mbox{Ch}_S(M\otimes_RS).$$
\end{prop}

\begin{proof} Let $\fp\in\Spec(R)$ such that $\mbox{Ch}_R(M)= \depth R_{\fp}-\width_{R_{\fp}}(M_{\fp})$. Let
$\fq\in\Spec(S)$ contain $\fp S$ minimally. Since $R\rightarrow S$
is a flat local homomorphism we have $\fp=\fq\cap R$. Hence:
\begin{align*}
\mbox{Ch}_R(M)= & \depth R_{\fp}-\width_{R_{\fp}}(M_{\fp}) \\[1ex]
               = & \depth S_{\fq}-\width_{S_{\fq}}(M_{\fp}\otimes_{R_{\fp}}S_{\fq}) \\[1ex]
               = & \depth S_{\fq}-\width_{S_{\fq}}(M\otimes_RS)_{\fq} \\[1ex]
               \le & \mbox{Ch}_S(M\otimes_RS),
\end{align*}
in which the second equality holds by Lemma \ref{CH41} and the fact
that $R_{\fp}\rightarrow S_{\fq}$ has Artinian closed fibre.
\end{proof}

\begin{prop} \label{CH5} Let $Q\rightarrow S$ be a $\CM$-deformation,
and $Y$ be a homologically bounded $S$-complex. Then
$$\mbox{Ch}_S(Y)\le\mbox{Ch}_Q(Y)-\gfd_Q(S).$$
\end{prop}

\begin{proof} Choose a prime ideal $\fp$ of $S$ such that the first
equality below holds. Let $\fq$ be the inverse image of $\fp$ in
$Q$. Therefore there is an isomorphism $Y_{\fp}\cong Y_{\fq}$ of
$Q_{\fq}$-complexes and a $\CM$-deformation $Q_{\fq}\rightarrow
S_{\fp}$.  Hence
\begin{align*}
\mbox{Ch}_S(Y)= & \depth S_{\fp}-\width_{S_{\fp}}(Y_{\fp}) \\[1ex]
        = & \depth_{Q_{\fq}} S_{\fp}-\width_{Q_{\fq}}(Y_{\fp}) \\[1ex]
        = & \depth Q_{\fq}-\gfd_{Q_{\fq}} S_{\fp}-\width_{Q_{\fq}}(Y_{\fp}) \\[1ex]
        \le & \mbox{Ch}_Q(Y)-\gfd_{Q_{\fq}}(S_{\fp}) \\[1ex]
        = & \mbox{Ch}_Q(Y)-\gfd_Q(S).
\end{align*}
The second equality holds since $Q_{\fq}\rightarrow S_{\fp}$ is
surjective; the third equality holds by Auslander-Bridger formula
\cite{AB}; the fourth equality is due to the
$\mbox{G}$-perfectness assumption of $S$ over $Q$.
\end{proof}

\begin{thm} \label{T} Let $M$ be a homologically bounded $R$-complex. Then there is the
inequality
$$\mbox{Ch}_R(M)\le\CMid_R(M).$$
\end{thm}

\begin{proof} We can assume that $\CMid_R(M)<\infty$. Choose a $\CM$-quasi-deformation
$R\rightarrow R'\leftarrow Q$, such that $\CMid_R(M)=\gid_Q
(M\otimes_R R')-\gfd_Q(R')$. Hence we have
\begin{align*}
\CMid_R(M)= & \gid_Q (M\otimes_R R')-\gfd_Q(R') \\[1ex]
        = & \mbox{Ch}_Q(M\otimes_R R')-\gfd_Q(R') \\[1ex]
        \ge & \mbox{Ch}_{R'}(M\otimes_R R')\geq \mbox{Ch}_R(M),
\end{align*}
in which the second equality comes by \cite[Theorem 2.2]{CS10}, and inequalities follow
Propositions \ref{CH5} and \ref{CH4} respectively.
\end{proof}

\begin{cor}\label{CH1} Let $M$ be a homologically bounded $R$-complex. Then there are
inequalities
$$\mbox{Ch}_R(M)\le\CMid_R(M)\le\gid_R(M),$$ such that if $\gid_R(M)<\infty$, then $\gid_R(M)=\CMid_R(M)$.
\end{cor}

\begin{proof} The inequalities hold by Theorem \ref{T} and Remark \ref{3.2}(2). And if $\gid_R(M)<\infty$, then the equality holds by
\cite[Theorem 2.2]{CS10}.
\end{proof}

\begin{cor}  Let $M$ be a homologically finite $R$-complex such that $\CMid_R(M)$ is finite. Then
\begin{align*}
\CMid_R(M)= & \mbox{Ch}_R(M)=\depth R-\inf(M) \\[1ex]
        \le & \dim R-\inf(M).
\end{align*}

\end{cor}

\begin{proof} By Proposition \ref{iart} there is a $\CM$-quasi-deformation $R\rightarrow R'\leftarrow Q$ such that the closed fibre of $R\rightarrow R'$ is Artinian and the first equality below holds. So that
\begin{align*}
\CMid_R(M)=&\gid_Q(M\otimes_RR') -\gfd_Q(R')\\[1ex]
        = &\depth Q-\inf(M\otimes_RR') -\depth Q+\depth R' \\[1ex]
        = & \depth R'-\inf(M\otimes_RR')  \\[1ex]
        = & \depth R-\inf(M).
\end{align*}
The second equality holds by \cite[Corollary 2.3]{CS10} and the Auslander-Bridger formula \cite{AB}, while the lase equality holds, because the closed fiber of $R\to R'$ is Artinian and \cite[Proposition 1.2.16]{BH}.

Now by Theorem \ref{T}, $\depth R-\inf(M)\le\mbox{Ch}_R(M)\le\CMid_R(M)=\depth R-\inf(M)$. Therefore $\CMid_R(M)=\mbox{Ch}_R(M)=\depth R-\inf(M)$.
\end{proof}

In concluding, recall that there are notions of Cohen-Macaulay
projective dimension, Cohen-Macaulay flat dimension and Cohen-Macaulay injective dimension of Holm and J{\o}rgensen, which are different with our Definition \ref{d}.

\begin{defn}\label{C}{\em (cf., \cite[Definition 2.3]{HJ}) Let $(R,\fm)$ be a local ring. For each homologically bounded $R$-complex $M$, the Cohen-Macaulay
projective, flat and injective dimension, of $M$ is defined as, respectively,
$$\CM\pd_R(M):=\inf\{\gpd_{R\ltimes C}(M)\mid C\text{ is a semidualizing module}\}$$
$$\CM\fd_R(M):=\inf\{\gfd_{R\ltimes C}(M)\mid C\text{ is a semidualizing module}\}$$
$$\CM\id_R(M):=\inf\{\gid_{R\ltimes C}(M)\mid C\text{ is a semidualizing module}\}.$$
Here $R\ltimes C$ denotes the trivial extension ring of $R$ by $C$; it is the $R$-module $R \oplus C$
equipped with the multiplication $(r,c)(r',c')=(rr',rc'+r'c)$.
}
\end{defn}

\begin{rem}{\em (1) For each homologically bounded $R$-complex $M$, we have
$$\CMpd_R(M)\le\CM\pd_R(M)$$
$$\CMfd_R(M)\le\CM\fd_R(M)$$
$$\CMid_R(M)\le\CM\id_R(M).$$
More precisely, assume that $\CM\pd_R(M)<\infty$ and choose a semidualizing $R$-module $C$ such that $\CM\pd_R(M)=\gpd_{R\ltimes C}(M)$. Then by \cite[Lemma 3.6]{G}, we have the $\CM$-quasi-deformation $R\to R\stackrel{\tau}{\leftarrow} Q$ where $Q:=R\ltimes C$ and $\tau(r,c)=r$, such that $\gd_Q(R)=0$. Thus we obtain $$\gpd_Q(M\otimes_RR)-\gfd_Q(R)=\gpd_{R\ltimes C}(M).$$ This shows the first inequality. The proof of the other two inequalities are the same as the first one.

(2) The finiteness of the Cohen-Macaulay homological dimensions in Definition \ref{C}, characterize
Cohen-Macaulay rings admitting a canonical module \cite[Theorem 5.1]{HJ}.

(3) Assume that $(R,\fm,k)$ is a Cohen-Macaulay ring, not admitting a canonical module (e.g., see \cite{FR70} for such an example). Then $\CMpd_R(k)<\infty$ (and, $\CMfd_R(k)<\infty$, $\CMid_R(k)<\infty$) but $\CM\pd_R(k)=\infty$ (and, $\CM\fd_R(k)=\infty$, $\CM\id_R(k)=\infty$).
 }
\end{rem}

\begin{lem}\label{t} Assume that $C$ is a semidualizing $R$-module and let $M$ be a homologically bounded $R$-complex. Consider $M$ as a $R\ltimes C$-complex via the natural surjection $\tau:R\ltimes C\to R$.
\begin{itemize}
  \item[(1)] If $\gfd_{R\ltimes C}(M)<\infty$, then $\gfd_{R\ltimes C}(M)=\rfd_R(M)$.
  \item[(2)] If $\gid_{R\ltimes C}(M)<\infty$, then $\gid_{R\ltimes C}(M)=\mbox{Ch}_R(M)$.
\end{itemize}
\end{lem}
\begin{proof} Note that $\Spec(R\ltimes C)=\{\fp\ltimes C\mid \fp\in\Spec(R)\}$ and $(R\ltimes C)_{\fp\ltimes C}\cong R_{\fp}\ltimes C_{\fp}$ by \cite[Exercise 6.2.12]{BS}. Let $L$ be an $R$-module which is an $R\ltimes C$-module via the surjection $\tau:R\ltimes C\to R$, and let $\fp$ be a prime ideal of $R$. Then $\varphi:L_{\fp\ltimes C}\to L_{\fp}$ sending $l/(r,c)$ to $l/r$ is an $R_{\fp}$-isomorphism. By \cite[Theorem 8.8]{IS} we have the first equality below.
\begin{align*}
\gfd_{R\ltimes C}(M)= & \sup\{\depth (R\ltimes C)_{\fp\ltimes C} - \depth_{(R\ltimes C)_{\fp\ltimes C}}(M_{\fp\ltimes C})\mid \fp\in\Spec(R)\}\\[1ex]
        = & \sup\{\depth_{R_{\fp}\ltimes C_{\fp}} (R_{\fp}\ltimes C_{\fp}) - \depth_{R_{\fp}\ltimes C_{\fp}}(M_{\fp\ltimes C})\mid \fp\in\Spec(R)\}\\[1ex]
        = & \sup\{\depth_{R_{\fp}} (R_{\fp}\oplus C_{\fp}) - \depth_{R_{\fp}}(M_{\fp})\mid \fp\in\Spec(R)\}\\[1ex]
        = & \sup\{\depth R_{\fp} - \depth_{R_{\fp}}(M_{\fp})\mid \fp\in\Spec(R)\}\\[1ex]
        = & \rfd_R(M).
\end{align*}
The third equality holds since there is a surjection $R_{\fp}\ltimes C_{\fp}\to R_{\fp}$ and \cite[Proposition 5.2(1)]{i99}. The fourth equality uses $$\depth_{R_{\fp}} (R_{\fp}\oplus C_{\fp})=\min\{\depth R_{\fp},\depth_{R_{\fp}}(C_{\fp})\}=\depth R_{\fp}.$$

The proof of (2) is the same as (1) using \cite[Theorem 2.2]{CS10} instead of \cite[Theorem 8.8]{IS}, and Lemma \ref{CH2}, instead of \cite[Proposition 5.2(1)]{i99}.
\end{proof}

\begin{cor}\label{U} Let $M$ be a homologically bounded $R$-complex.
\begin{itemize}
  \item[(1)] If $\CM\fd_R(M)<\infty$, then $\CM\fd_R(M)=\CMfd_R(M)$.
  \item[(2)] If $\CM\id_R(M)<\infty$, then $\CM\id_R(M)=\CMid_R(M)$.
\end{itemize}
\end{cor}

\begin{proof} Note that there are the inequalities $$\rfd_R(M)\le\CMfd_R(M)\le\CM\fd_R(M)=\rfd_R(M)$$ (resp., $\mbox{Ch}_R(M)\le\CMid_R(M)\le\CM\id_R(M)=\mbox{Ch}_R(M)$) by Corollary \ref{rfdcmfd} (resp., Theorem \ref{T}), and Lemma \ref{t}.
\end{proof}

\noindent {\bf Acknowledgement.} The authors would like to thank the referee for his/her careful reading of the manuscript and several comments
which greatly improved the paper. Parviz Sahandi would like to thank Sean Sather-Wagstaff for comments on an earlier version of this paper. Part of this work was completed while Siamak Yassemi was visiting the Institut des Hautes Etudes Scientifiques (IHES) in Bures-sur-Yvette, France. He wishes to express his gratitude to the Institute for its warm hospitality and for providing a stimulating research environment.

\end{document}